\documentclass[a4paper]{amsart}
\usepackage{amssymb}
\usepackage{amsmath}
\usepackage{amsfonts}
\usepackage{amsthm}
\usepackage{color}
\usepackage{graphicx,url}
\usepackage{enumerate}

\usepackage{lineno,hyperref}
\usepackage{lmodern,multicol}
\usepackage{thm-restate}

 \usepackage[capitalise,noabbrev]{cleveref}

\usepackage{subcaption}

\usepackage{amsthm}
\usepackage{amsmath,amssymb,mathrsfs,dsfont}
\usepackage{tikz}
\usepackage[T1]{fontenc}
\usepackage[utf8]{inputenc}

\usepackage{graphicx}

\newtheorem{theorem}{\bfseries Theorem}[section]

\newtheorem{lemma}[theorem]{\bfseries Lemma}
\theoremstyle{definition}
\newtheorem{definition}[theorem]{\bfseries Definition}
\newtheorem{example}[theorem]{\bfseries Example}

\usetikzlibrary{arrows, automata,chains,fit,shapes}
\tikzstyle{decision} = [diamond, draw, fill=blue!20,
text width=4.5em, text badly centered, node distance=2.5cm, inner sep=0pt]
\tikzstyle{block} = [rectangle, draw, fill=blue!20,
text width=5em, text centered, rounded corners, minimum height=4em]
\tikzstyle{line} = [draw, thick, color=black!50, -latex']
\tikzstyle{cloud} = [draw, ellipse,fill=red!20, node distance=2.5cm,
minimum height=2em]

\usepackage{xcolor}

\usepackage{ulem}

\newcommand\N{\mathbb N}

\newcommand\unbar{\mathrm{unbar}}
\newcommand\removebar{\mathrm{removebar}}
\newcommand\red{\text{red}}

\newcommand\Av{\text{Av}}

\renewcommand{\geq}{\geqslant} \renewcommand{\leq}{\leqslant}  \renewcommand{\le}{\leqslant}

\newcommand\containedin{\le}
\newcommand\subperm{<_{\text{subperm}}}

\begin{document}
 \normalem

\title[$k$-pop stack sortable permutations and $2$-avoidance]{$k$-pop stack sortable permutations\\ and $2$-avoidance}

\author{Murray Elder, Yoong Kuan Goh}
\address{School of Mathematical and Physical Sciences, University of Technology Sydney, Ultimo NSW 2007, Australia}
\email{murrayelder,gohyoongkuan@gmail.com}
\thanks{Research supported by Australian Research Council grant DP160100486. The second author was supported by the Australian Government Research Training Program.}

\date{\today}

\begin{abstract}

	We consider 
	permutations  sortable by $k$ passes through a deterministic pop stack. We show that for any 
	 $k\in\N$
	 the set is characterised by finitely many 
	 patterns, answering a question of Claesson and Gu\dh{}mundsson.

	  Our characterisation demands a more precise definition  than in previous literature of what it means for a permutation to  avoid a set of barred and unbarred patterns. We propose a new notion called {\em $2$-avoidance}. \end{abstract}

%\classification[2010]{05A05} 

\maketitle

\section{Introduction}
A {\em pop stack}   is a sorting device which operates as follows: at each step it can either push one token from the input stream onto the top of the stack, or else  pop the {\em entire stack contents}  to the output stream. We consider the tokens to be distinct real numbers 
  with the usual ordering.
  A  {\em deterministic} pop stack always performs the push move unless the token 
 on the top of the stack is larger in value that the token to be pushed from the input, or if there is no further input.
See for example Figure~\ref{fig:deterministic-pop stack}.  For convenience from now on we assume tokens are integers and write a sequence $2,1,3$ as $213$ when tokens are single digits.

\begin{figure}[h!]
    \centering

		\begin{tikzpicture}[scale=.8, node distance = 2.1cm, auto]
		\begin{node} (root)
		{
			\begin{tikzpicture}[thick,scale=0.4, every node/.style={scale=0.6}]
			\draw[thick] (0,1.7) -- (0,-2.5) -- (1,-2.5) -- (1,1.7);
			\draw[ ->] (1.6,2.4) -- node[above] {} (0.66,2.4) -- (0.66,1.7);
			\draw (2.5,2.4) node {$213$};
			\draw[ <-] (-0.6,2.4) -- node[above] {} (0.33,2.4) -- (0.33,1.7);	
			\end{tikzpicture}
		};
		\end{node}
	
	\node[right of=root, node distance=2.1cm] (root_2)
	{
		\begin{tikzpicture}[thick,scale=0.4, every node/.style={scale=0.6}]
		\draw[thick] (0,1.7) -- (0,-2.5) -- (1,-2.5) -- (1,1.7);
		\draw[ ->] (1.6,2.4) -- node[above] {} (0.66,2.4) -- (0.66,1.7);
		\draw (2.5,2.4) node {$13$};
		\draw[ <-] (-0.6,2.4) -- node[above] {} (0.33,2.4) -- (0.33,1.7);	
		\draw (0.5,-1.7) node {$2$};
		\end{tikzpicture}
	};

	\node[right of=root_2, node distance=2.1cm] (root_3)
	{
		\begin{tikzpicture}[thick,scale=0.4, every node/.style={scale=0.6}]
		\draw[thick] (0,1.7) -- (0,-2.5) -- (1,-2.5) -- (1,1.7);
		\draw[ ->] (1.6,2.4) -- node[above] {} (0.66,2.4) -- (0.66,1.7);
		\draw (2.1,2.4) node {$3$};
		\draw[ <-] (-0.6,2.4) -- node[above] {} (0.33,2.4) -- (0.33,1.7);	
		\draw (0.5,-1.7) node {$2$};
		\draw (0.5,-1.2) node {$1$};
		\end{tikzpicture}
	};

	\node[right of=root_3, node distance=2.1cm] (root_4)
	{
		\begin{tikzpicture}[thick,scale=0.4, every node/.style={scale=0.6}]
		\draw[thick] (0,1.7) -- (0,-2.5) -- (1,-2.5) -- (1,1.7);
		\draw[ ->] (1.6,2.4) -- node[above] {} (0.66,2.4) -- (0.66,1.7);
		\draw (2.0,2.4) node {};
		\draw[ <-] (-0.6,2.4) -- node[above] {} (0.33,2.4) -- (0.33,1.7);	
		\draw (2.1,2.4) node {$3$};
		\draw (-1.2,2.4) node {};
				\draw (-1.3,2.4) node {$12$};
		\end{tikzpicture}
	};

	\node[right of=root_4, node distance=2.1cm] (root_5)
	{
		\begin{tikzpicture}[thick,scale=0.4, every node/.style={scale=0.6}]
		\draw[thick] (0,1.7) -- (0,-2.5) -- (1,-2.5) -- (1,1.7);
		\draw[ ->] (1.6,2.4) -- node[above] {} (0.66,2.4) -- (0.66,1.7);
		\draw[ <-] (-0.6,2.4) -- node[above] {} (0.33,2.4) -- (0.33,1.7);	
		\draw (0.5,-1.2) node {$3$};
		\draw (-1.3,2.4) node {$12$};
		\end{tikzpicture}
	};

	\node[right of=root_5, node distance=2.1cm] (root_6)
	{
		\begin{tikzpicture}[thick,scale=0.4, every node/.style={scale=0.6}]
		\draw[thick] (0,1.7) -- (0,-2.5) -- (1,-2.5) -- (1,1.7);
		\draw[ ->] (1.6,2.4) -- node[above] {} (0.66,2.4) -- (0.66,1.7);
		\draw[ <-] (-0.6,2.4) -- node[above] {} (0.33,2.4) -- (0.33,1.7);	
		\draw (-1.3,2.4) node {$123$};
		\end{tikzpicture}
	};
	
	\path [line] (root) -- (root_2);
	\path [line] (root_2) -- (root_3);
	\path [line] (root_3) -- (root_4);
	\path [line] (root_4) -- (root_5);
	\path [line] (root_5) -- (root_6);

	 \end{tikzpicture}

\caption{Sorting $213$ using a deterministic pop stack}
\label{fig:deterministic-pop stack}
\end{figure}
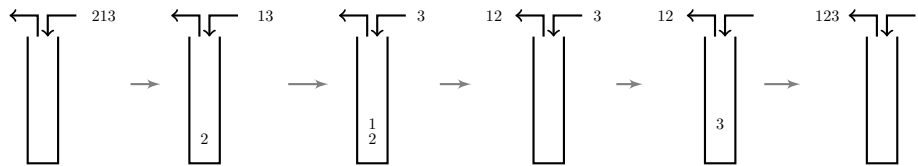

Observe that by definition the stack remains ordered from smallest on top to largest on the bottom during the operation of a deterministic pop stack. 

A permutation (ordered sequence of distinct real numbers) can be sorted by $k$ passes through a deterministic pop stack if after repeating the procedure $k$ times, the sequence is ordered from smallest to largest. For example, 
$41352$ can be sorted by two passes  (\cref{fig:2-pass pop stack}). 

\begin{figure}[ht!]

   \begin{subfigure}{1\textwidth}
       \centering
   \begin{tikzpicture}[scale=1, node distance = 2cm, auto]
		\begin{node} (root)
		{
			\begin{tikzpicture}[thick,scale=0.4, every node/.style={scale=0.6}]
			\draw[thick] (0,1.7) -- (0,-2.5) -- (1,-2.5) -- (1,1.7);
			\draw[ ->] (1.6,2.4) -- node[above] {} (0.66,2.4) -- (0.66,1.7);
			\draw (2.5,2.4) node {$41352$};
			\draw[ <-] (-0.6,2.4) -- node[above] {} (0.33,2.4) -- (0.33,1.7);	
			\end{tikzpicture}
		};
		\end{node}

	\node[right of=root, node distance=2.3cm] (root_3)
	{
		\begin{tikzpicture}[thick,scale=0.4, every node/.style={scale=0.6}]
		\draw[thick] (0,1.7) -- (0,-2.5) -- (1,-2.5) -- (1,1.7);
		\draw[ ->] (1.6,2.4) -- node[above] {} (0.66,2.4) -- (0.66,1.7);
		\draw (2.1,2.4) node {$352$};
		\draw[ <-] (-0.6,2.4) -- node[above] {} (0.33,2.4) -- (0.33,1.7);	
		\draw (0.5,-1.7) node {$4$};
		\draw (0.5,-1.2) node {$1$};
		\end{tikzpicture}
	};

	\node[right of=root_3, node distance=2.3cm] (root_4)
	{
		\begin{tikzpicture}[thick,scale=0.4, every node/.style={scale=0.6}]
		\draw[thick] (0,1.7) -- (0,-2.5) -- (1,-2.5) -- (1,1.7);
		\draw[ ->] (1.6,2.4) -- node[above] {} (0.66,2.4) -- (0.66,1.7);
		\draw (2.0,2.4) node {$52$};
		\draw[ <-] (-0.6,2.4) -- node[above] {} (0.33,2.4) -- (0.33,1.7);	
		\draw (0.5,-1.7) node {$3$};
		\draw (-1.2,2.4) node {$14$};
		\end{tikzpicture}
	};

	\node[right of=root_4, node distance=2.3cm] (root_5)
	{
		\begin{tikzpicture}[thick,scale=0.4, every node/.style={scale=0.6}]
		\draw[thick] (0,1.7) -- (0,-2.5) -- (1,-2.5) -- (1,1.7);
		\draw[ ->] (1.6,2.4) -- node[above] {} (0.66,2.4) -- (0.66,1.7);
		\draw[ <-] (-0.6,2.4) -- node[above] {} (0.33,2.4) -- (0.33,1.7);	
		\draw (0.5,-1.7) node {$5$};
		\draw (0.5,-1.2) node {$2$};
		\draw (-1.3,2.4) node {$143$};
		\end{tikzpicture}
	};

	\node[right of=root_5, node distance=2.3cm] (root_6)
	{
		\begin{tikzpicture}[thick,scale=0.4, every node/.style={scale=0.6}]
		\draw[thick] (0,1.7) -- (0,-2.5) -- (1,-2.5) -- (1,1.7);
		\draw[ ->] (1.6,2.4) -- node[above] {} (0.66,2.4) -- (0.66,1.7);
		\draw[ <-] (-0.6,2.4) -- node[above] {} (0.33,2.4) -- (0.33,1.7);	
		\draw (-1.8,2.4) node {$14325$};
		\end{tikzpicture}
	};
	
	\path [line] (root) -- (root_3);
	\path [line] (root_3) -- (root_4);
	\path [line] (root_4) -- (root_5);
	\path [line] (root_5) -- (root_6);

\end{tikzpicture}
    \caption{First pass.}
 \end{subfigure}

  \bigskip
  \begin{subfigure}{1\textwidth}
    \centering
\begin{tikzpicture}[scale=2, node distance = 2cm, auto]
\begin{node} (root)
{
	\begin{tikzpicture}[thick,scale=0.4, every node/.style={scale=0.6}]
	\draw[thick] (0,1.7) -- (0,-2.5) -- (1,-2.5) -- (1,1.7);
	\draw[ ->] (1.6,2.4) -- node[above] {} (0.66,2.4) -- (0.66,1.7);
	\draw (2.5,2.4) node {$14325$};
	\draw[ <-] (-0.6,2.4) -- node[above] {} (0.33,2.4) -- (0.33,1.7);	
	\end{tikzpicture}
};
\end{node}

\node[right of=root, node distance=2.3cm] (root_2)
{
	\begin{tikzpicture}[thick,scale=0.4, every node/.style={scale=0.6}]
	\draw[thick] (0,1.7) -- (0,-2.5) -- (1,-2.5) -- (1,1.7);
	\draw[ ->] (1.6,2.4) -- node[above] {} (0.66,2.4) -- (0.66,1.7);
	\draw (2.4,2.4) node {$4325$};
	\draw[ <-] (-0.6,2.4) -- node[above] {} (0.33,2.4) -- (0.33,1.7);	
	\draw (0.5,-1.7) node {$1$};
	\end{tikzpicture}
};

\node[right of=root_2, node distance=2.3cm] (root_3)
{
	\begin{tikzpicture}[thick,scale=0.4, every node/.style={scale=0.6}]
	\draw[thick] (0,1.7) -- (0,-2.5) -- (1,-2.5) -- (1,1.7);
	\draw[ ->] (1.6,2.4) -- node[above] {} (0.66,2.4) -- (0.66,1.7);
	\draw (2.2,2.4) node {$325$};
	\draw[ <-] (-0.6,2.4) -- node[above] {} (0.33,2.4) -- (0.33,1.7);	
	\draw (0.5,-1.7) node {$4$};
	\draw (-1.0,2.4) node {$1$};
	\end{tikzpicture}
};

\node[right of=root_3, node distance=2.3cm] (root_4)
{
	\begin{tikzpicture}[thick,scale=0.4, every node/.style={scale=0.6}]
	\draw[thick] (0,1.7) -- (0,-2.5) -- (1,-2.5) -- (1,1.7);
	\draw[ ->] (1.6,2.4) -- node[above] {} (0.66,2.4) -- (0.66,1.7);
	\draw (1.8,2.4) node {$5$};
	\draw[ <-] (-0.6,2.4) -- node[above] {} (0.33,2.4) -- (0.33,1.7);	
	\draw (0.5,-1.7) node {$4$};
	\draw (0.5,-1.2) node {$3$};
	\draw (0.5,-0.7) node {$2$};
	\draw (-1.2,2.4) node {$1$};
	\end{tikzpicture}
};

\node[right of=root_4, node distance=2.3cm] (root_5)
{
	\begin{tikzpicture}[thick,scale=0.4, every node/.style={scale=0.6}]
	\draw[thick] (0,1.7) -- (0,-2.5) -- (1,-2.5) -- (1,1.7);
	\draw[ ->] (1.6,2.4) -- node[above] {} (0.66,2.4) -- (0.66,1.7);
	\draw[ <-] (-0.6,2.4) -- node[above] {} (0.33,2.4) -- (0.33,1.7);	
	\draw (0.5,-1.7) node {$5$};
	\draw (-1.4,2.4) node {$1234$};
	\end{tikzpicture}
};

\path [line] (root) -- (root_2);
\path [line] (root_2) -- (root_3);
\path [line] (root_3) -- (root_4);
\path [line] (root_4) -- (root_5);

\end{tikzpicture}

\caption{Second pass.}

\end{subfigure}

\caption{Sorting $41352$ with a 2-pass pop stack}
\label{fig:2-pass pop stack}

\end{figure}
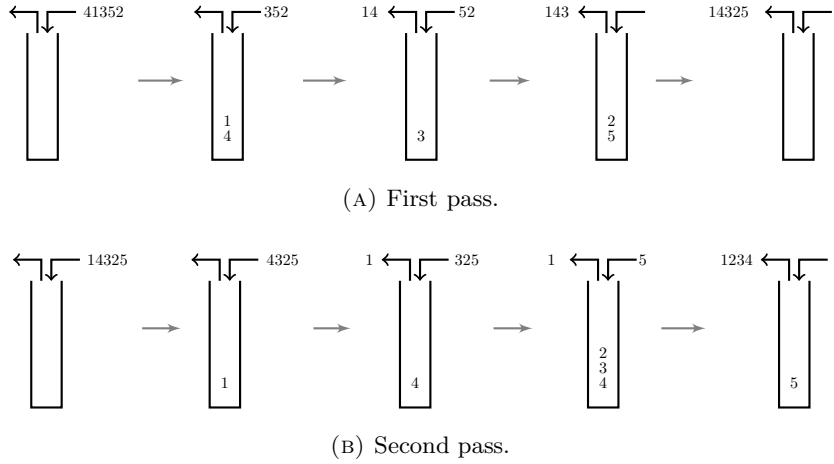
    
Let $p_1(\alpha)$ denote  the sequence obtained by passing a sequence $\alpha$ through a deterministic pop stack once, and define
$p_k(\alpha)=p_{k-1}(p_1(\alpha))$.
For example $p_1(41352)=14325$ and 
$p_2(41352)=12345$.
We say $\alpha$ is {\em $k$-pass deterministic pop stack sortable}, or {\em $k$-pop stack sortable} for short, if  $p_k(\alpha)$ is an increasing sequence.

In 1981 Avis and Newborn  characterised permutations sorted by  a single pass through a pop stack \cite{MR624050}, initiating the study of pop stack sorting. 
Specifically they showed a permutation can be sorted in one pass if and only if it avoids $231$ and $312$ in the usual sense of pattern avoidance (see \cref{sec:prelim} for precise definitions).
Pudwell and Smith characterised permutations sorted by  2 passes, in terms of avoiding a set of six usual patterns and two special {\em barred} patterns (defined in Subsection~\ref{subsec:PB} 
below), and computed a rational generating function for the number of  such permutations \cite{RebeccaTwoPass}.
  Claesson and Gu\dh{}mundsson then computed a rational generating function for permutations sorted by any finite number of passes \cite{MR3940979}, and  
  asked whether a ``useful permutation pattern characterization of the $k$-pop stack-sortable permutations'' exists for $k\geq 3$.

In this paper we provide such a characterisation. To do so, we 
  realised the current notions of what it means to avoid a set of barred and unbarred patterns would not suffice (see Subsection~\ref{subsec:PB}). We therefore 
 introduce a new notion  called {\em $2$-avoidance}, which could be of independent interest. 

We make the following observations. First, our present result is in contrast to the usual (nondeterministic) stacks-in-series model where in many cases  no finite pattern-avoidance characterisation is possible due to the existence of infinite antichains  \cite{MR3788054,MurphyPhD}.
Second, the operation of a pop stack is related to classical sorting:   ``bubble-sort'' is exactly sorting by arbitrarily  many passes through a pop stack of depth $2$.  
Third, 
pop stacks are a natural model for genome rearrangement \cite{RebeccaTwoPass}.

\section{Preliminaries}\label{sec:prelim}

Define a
 {\em permutation} to be a sequence of distinct real numbers, written as $\alpha=a_1a_2\dots a_r$ (we continue the convention to write sequences without commas). 
The {\em reduced form} of a permutation $\alpha$, denoted $\red(\alpha)$, is the permutation obtained by replacing the $i$th smallest entry of $\alpha$ by the integer $i$. A permutation in reduced form is called {\em reduced}. We denote the set of all reduced   permutations by $S^\infty$.
Two permutations  $\alpha=a_1\dots a_r$ and $\beta=b_1\dots b_s$
 are {\em order-isomorphic}, denoted $\alpha\sim\beta$, if they have the same reduced form.
 For example $253$ and $132$ are order-isomorphic. 
 In general, we will write permutations in their reduced form.

A {\em subpermutation} of $\alpha=a_1\dots a_r$ is a sequence $a_{i_1}\dots a_{i_s}$
where $1\leq i_1<\dots <i_s\leq r$, 
while  a   {\em factor}  is a sequence  $a_{i_1}\dots a_{i_s}$ where $i_{j+1}=i_j+1$.
If $\alpha,\beta$ are two permutations, we say $\beta$ {\em contains} $\alpha$ if some subpermutation of $\beta$ is order-isomorphic to $\alpha$. We use the notation $\alpha\containedin \beta$ if  $\beta$ {contains} $\alpha$, and the notation $\alpha \subperm \beta$ if $\alpha$ is a subpermutation of $\beta$.
We say
$\beta$ {\em avoids} $\alpha$ if no  subpermutation of $\beta$ is order-isomorphic to $\alpha$. For example $54123$ contains $312$ and  avoids  $231$.  For any set of  permutations $F$, let $\Av(F)\subseteq S^\infty$ denote the  set of all permutations that simultaneously avoid every $\alpha\in F$.

Knuth famously observed that a permutation can be sorted by passing it through a single infinite stack if and only if it avoids $231$ \cite{Knuth}.
Avis and Newborn proved a permutation can be sorted by passing it through an infinite pop stack once if and only if it avoids both $231$ and $312$. However, for multiple passes through a pop stack, the situation arises that some permutation cannot be sorted, while a longer permutation containing it can, so the usual notion of  pattern avoidance is not useful for characterising permutations in this context (in other words, for $k\geq 2$, $k$-pass pop stack sortable permutations are not a closed class with respect to usual pattern avoidance). As a concrete example,  Figure~\ref{fig:2-pass pop stack:3214} shows that $3241$ is not $2$-pass pop stack sortable, whereas $41352$ is (as demonstrated by Figure~\ref{fig:2-pass pop stack}), and contains $3241$.

\begin{figure}[h!]
	 \begin{subfigure}[t]{0.8\textwidth}
    	\centering

		\begin{tikzpicture}[scale=1, node distance = 2cm, auto]
		\begin{node} (root)
		{
			\begin{tikzpicture}[thick,scale=0.4, every node/.style={scale=0.6}]
			\draw[thick] (0,1.7) -- (0,-2.5) -- (1,-2.5) -- (1,1.7);
			\draw[ ->] (1.6,2.4) -- node[above] {} (0.66,2.4) -- (0.66,1.7);
			\draw (2.5,2.4) node {$3241$};
			\draw[ <-] (-0.6,2.4) -- node[above] {} (0.33,2.4) -- (0.33,1.7);	
			\end{tikzpicture}
		};
	\end{node}
	
	\node[right of=root, node distance=2.3cm] (root_2)
	{
		\begin{tikzpicture}[thick,scale=0.4, every node/.style={scale=0.6}]
		\draw[thick] (0,1.7) -- (0,-2.5) -- (1,-2.5) -- (1,1.7);
		\draw[ ->] (1.6,2.4) -- node[above] {} (0.66,2.4) -- (0.66,1.7);
		\draw (2.5,2.4) node {$41$};
		\draw[ <-] (-0.6,2.4) -- node[above] {} (0.33,2.4) -- (0.33,1.7);	
		\draw (0.5,-1.7) node {$3$};
		\draw (0.5,-1.2) node {$2$};
		\end{tikzpicture}
	};
	
	\node[right of=root_2, node distance=2.3cm] (root_3)
	{
		\begin{tikzpicture}[thick,scale=0.4, every node/.style={scale=0.6}]
		\draw[thick] (0,1.7) -- (0,-2.5) -- (1,-2.5) -- (1,1.7);
		\draw[ ->] (1.6,2.4) -- node[above] {} (0.66,2.4) -- (0.66,1.7);
		\draw (2.1,2.4) node {$41$};
		\draw[ <-] (-0.6,2.4) -- node[above] {} (0.33,2.4) -- (0.33,1.7);	
		\draw (-1.0,2.4) node {$23$};
		\end{tikzpicture}
	};
	
	\node[right of=root_3, node distance=2.3cm] (root_4)
	{
		\begin{tikzpicture}[thick,scale=0.4, every node/.style={scale=0.6}]
		\draw[thick] (0,1.7) -- (0,-2.5) -- (1,-2.5) -- (1,1.7);
		\draw[ ->] (1.6,2.4) -- node[above] {} (0.66,2.4) -- (0.66,1.7);
		\draw[ <-] (-0.6,2.4) -- node[above] {} (0.33,2.4) -- (0.33,1.7);	
		\draw (0.5,-1.7) node {$4$};
		\draw (0.5,-1.2) node {$1$};
		\draw (-1.2,2.4) node {$23$};
		\end{tikzpicture}
	};

	\node[right of=root_4, node distance=2.3cm] (root_5)
	{
		\begin{tikzpicture}[thick,scale=0.4, every node/.style={scale=0.6}]
		\draw[thick] (0,1.7) -- (0,-2.5) -- (1,-2.5) -- (1,1.7);
		\draw[ ->] (1.6,2.4) -- node[above] {} (0.66,2.4) -- (0.66,1.7);
		\draw[ <-] (-0.6,2.4) -- node[above] {} (0.33,2.4) -- (0.33,1.7);	
		\draw (-1.8,2.4) node {$2314$};
		\end{tikzpicture}
	};
	
	\path [line] (root) -- (root_2);
	\path [line] (root_2) -- (root_3);
	\path [line] (root_3) -- (root_4);
	\path [line] (root_4) -- (root_5);

\end{tikzpicture}
  \caption{First pass.}
  \end{subfigure}

  \bigskip
  \begin{subfigure}[t]{0.8\textwidth}
	\centering

		\begin{tikzpicture}[scale=1, node distance = 2cm, auto]
\begin{node} (root)
{
	\begin{tikzpicture}[thick,scale=0.4, every node/.style={scale=0.6}]
	\draw[thick] (0,1.7) -- (0,-2.5) -- (1,-2.5) -- (1,1.7);
	\draw[ ->] (1.6,2.4) -- node[above] {} (0.66,2.4) -- (0.66,1.7);
	\draw (2.5,2.4) node {$2314$};
	\draw[ <-] (-0.6,2.4) -- node[above] {} (0.33,2.4) -- (0.33,1.7);	
	\end{tikzpicture}
};
\end{node}

\node[right of=root, node distance=2.3cm] (root_2)
{
\begin{tikzpicture}[thick,scale=0.4, every node/.style={scale=0.6}]
\draw[thick] (0,1.7) -- (0,-2.5) -- (1,-2.5) -- (1,1.7);
\draw[ ->] (1.6,2.4) -- node[above] {} (0.66,2.4) -- (0.66,1.7);
\draw (2.5,2.4) node {$314$};
\draw[ <-] (-0.6,2.4) -- node[above] {} (0.33,2.4) -- (0.33,1.7);	
\draw (0.5,-1.7) node {$2$};
\end{tikzpicture}
};

\node[right of=root_2, node distance=2.3cm] (root_3)
{
\begin{tikzpicture}[thick,scale=0.4, every node/.style={scale=0.6}]
\draw[thick] (0,1.7) -- (0,-2.5) -- (1,-2.5) -- (1,1.7);
\draw[ ->] (1.6,2.4) -- node[above] {} (0.66,2.4) -- (0.66,1.7);
\draw (2.1,2.4) node {$4$};
\draw[ <-] (-0.6,2.4) -- node[above] {} (0.33,2.4) -- (0.33,1.7);	
\draw (0.5,-1.7) node {$3$};
\draw (0.5,-1.2) node {$1$};
\draw (-1.0,2.4) node {$2$};
\end{tikzpicture}
};

\node[right of=root_3, node distance=2.3cm] (root_4)
{
\begin{tikzpicture}[thick,scale=0.4, every node/.style={scale=0.6}]
\draw[thick] (0,1.7) -- (0,-2.5) -- (1,-2.5) -- (1,1.7);
\draw[ ->] (1.6,2.4) -- node[above] {} (0.66,2.4) -- (0.66,1.7);
\draw[ <-] (-0.6,2.4) -- node[above] {} (0.33,2.4) -- (0.33,1.7);	
\draw (0.5,-1.7) node {$4$};
\draw (-1.2,2.4) node {$213$};
\end{tikzpicture}
};

\node[right of=root_4, node distance=2.3cm] (root_5)
{
\begin{tikzpicture}[thick,scale=0.4, every node/.style={scale=0.6}]
\draw[thick] (0,1.7) -- (0,-2.5) -- (1,-2.5) -- (1,1.7);
\draw[ ->] (1.6,2.4) -- node[above] {} (0.66,2.4) -- (0.66,1.7);
\draw[ <-] (-0.6,2.4) -- node[above] {} (0.33,2.4) -- (0.33,1.7);	
\draw (-1.8,2.4) node {$2134$};
\end{tikzpicture}
};

\path [line] (root) -- (root_2);
\path [line] (root_2) -- (root_3);
\path [line] (root_3) -- (root_4);
\path [line] (root_4) -- (root_5);

\end{tikzpicture}

    \caption{Second pass.}
  \end{subfigure}

\caption{$3241$ is not a 2-pass pop stack sortable}
\label{fig:2-pass pop stack:3214}
\end{figure}
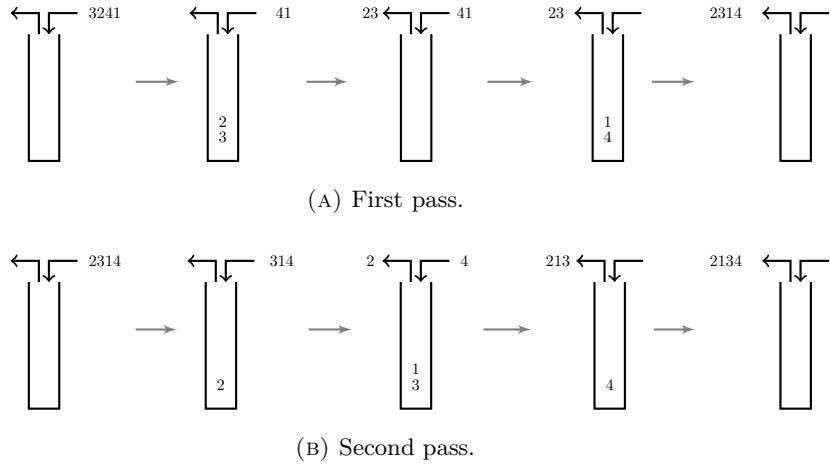

This leads us to the following notion.

\begin{definition}[$2$-containment]
	\label{2containment}
Let $\sigma$  be a permutation and $F, G\subseteq S^\infty$ be two sets of reduced permutations.
We say that $\sigma$ {\em $2$-contains} $(F,G)$ if there exists a subpermutation $\gamma$ of $\sigma$ such that 
\begin{itemize}\item[--] $\red(\gamma)\in F$ {and} 
\item[--] there is no 
$\delta\subperm \sigma$  such that 
 $\gamma\subperm \delta$ and $\red(\delta)\in G$.
\end{itemize}
\end{definition}
Informally we think of the set $G$ as patterns which can potentially {\em save} a permutation from being forbidden by $F$.

A permutation {\em $2$-avoids} $(F,G)$ if it does not $2$-contain $(F,G)$. Using propositional logic we can express this as follows.
\begin{definition}[$2$-avoidance]
	\label{2avoidance}
Let $\sigma$  be a permutation and $F, G\subseteq S^\infty$ be two sets of permutations.
We say that $\sigma$ {\em $2$-avoids} $(F,G)$ if for all  subpermutations $\gamma$ of $\sigma$, if $\red(\gamma)\in F$ then there exists 
$\delta\subperm \sigma$  such that 
 $\gamma\subperm \delta$ and $\red(\delta)\in G$.
\end{definition}
We denote the set of all permutations in $S^\infty$ which $2$-avoid 
$(F,G)$ by $\Av_2(F,G)$.

\begin{example}\label{eg:egSimple}
Let
$F=\{ 3241 \}, G=\{ 4{1}352 \}$. Then  
 $\sigma_1=143562$ has  the subpermutation $4352~\sim  3241$ which is not part of a longer subpermutation of $\sigma_1$  order-isomorphic to  $41352$, so $\sigma_1$ $2$-contains $(F,G)$. 
Now consider $\sigma_2=152463$ which 
 has subpermutation $5463\sim 3241$, however $5463$ is a subpermutation of $52463\sim 41352$,  so $5462$  is {\em saved} by $G$. Since there are no other subpermutations of $\sigma_2$ that are order-isomorphic to $3241$, we have that $\sigma_2$ $2$-avoids $(F,G)$.
\end{example}

\begin{example}\label{eg:egFactorial}
Let
$F=\{ 1 \}, G=\{ 12,21 \}$. Then $\Av_2(F,G)$ consists of {\em all} permutations in $S^\infty$ except the permutation of length 1. This is a rather extreme example, but shows that the growth of 
$2$-avoidance sets can be factorial.\end{example}

Using this notion, we can express the result of Pudwell and Smith as follows.

\begin{theorem}[Pudwell and Smith \cite{RebeccaTwoPass}]\label{thm:2pop}
The set of $2$-pass pop stack sortable permutations is equal to  \[\Av_2(\{ 2341, 3412, 3421 ,4123, 4231, 4312, 3241,  4132 \}, \{4{1}352,  413{5}2 \})\] 

\end{theorem}

At this point the reader might object and say that Pudwell and Smith's use of {\em barred patterns} (see below)
 is more efficient to describe this set, which it is, however for $3$-pop stack sortable permutations, the usual definition of avoiding sets of barred and unbarred patterns  will fail.

\subsection{The problem with barred  patterns 
}\label{subsec:PB}
A {\em barred pattern} is a permutation where certain entries (possibly none) are marked with a bar, for example $\beta= 4\bar{3}5\bar{1}2$, which  encodes two permutations, one called $\removebar(\beta)$ obtained by deleting entries marked with a bar, 
and  the other called $\unbar(\beta)$ obtained by removing bars. For example $\removebar(4\bar{3}5\bar{1}2)=452\sim 231$, and $\unbar(4\bar{3}5\bar{1}2)=4{3}5{1}2$.

 In   \cite[Definition 1.2.3]{Kitaev}
 a permutation $\sigma$ is said to avoid a barred pattern  $\beta$ if {each} occurrence of $\removebar(\beta)$ in $\sigma$ (if any) is a part of an occurrence of $\unbar(\beta)$ in $\sigma$. 
  There are two issues with this definition.
 Firstly, as written,  this does not agree with the usual pattern avoidance when $\beta$ itself has no bar tokens. For example, $\sigma=21$ obviously does not avoid itself in the usual sense of pattern avoidance, but if $\beta=21$ is considered as a barred pattern then $\sigma$ avoids $\beta$ since it has a subpermutation, $\kappa = 21=\removebar(\beta)$,  and $\kappa$ is part of an occurrence of $\unbar(\beta)=21$ in $\sigma$.

 Secondly and more seriously,  in applications such as \cite{ RebeccaTwoPass,WEST1993303} some set of permutations $S$ is characterised by being those permutations avoiding some list of barred and unbarred patterns, where, as we understand it, this  means that each permutation in $S$ must avoid every pattern individually. Explicitly, Tenner   \cite{Tenner} defines, for   $P$ a collection of barred and unbarred patterns,  $Av(P)$ to be the set of permutations simultaneously  avoiding all patterns in $P$:
 \[Av(P)=\bigcap_{p\in P}Av(p)\]
  
  However, this notion does not suffice to describe $3$-pop stack sortable permutations: one may verify that $32451$ is not $3$-pop stack sortable while both $4{6}3{1}572$ and $4{7}3{1}562$ are. 
  If we were to characterise 
  $3$-pop stack sortable permutations  as those avoiding a list containing 
$4\bar{6}3\bar{1}572$ and $ 4\bar{7}3\bar{1}562$, then we would be mistaken since $4731562$ does not avoid this list since it fails to avoid the barred pattern $4\bar{6}3\bar{1}572$. 
\Cref{2containment} says that even though some permutation may contain a subpermutation which is forbidden, it can be saved if it extends to another subpermutation which appears somewhere on the list $G$.
 For the applications in
 \cite{ RebeccaTwoPass, WEST1993303} 
  the sets of barred and unbarred patterns to be avoided have no ``overlap'': in the case of \cite{ RebeccaTwoPass}, the two barred patterns have $\removebar$ equal to $2341$ and $4312$ which are both different to the unbarred patterns in their list; and in the case of \cite{WEST1993303} there is only one barred pattern whose $\removebar$ is different to the unbarred pattern.

\subsection{Removing redundant patterns}\label{subsec:redundant}

In general, one cannot simply delete patterns from $F$  if they contain shorter elements of $F$, or delete patterns from $G$ if they are contained in longer elements of $G$, since $2$-containment involves a subtle interplay between the two sets. The following two examples demonstrate this.

\begin{example}\label{egA}
If 	$F=\{ 43251,3241 \}$ and
$G=\{4{1}352 \}$
then we claim that the pattern $43251\in F$ is not  redundant, even though it contains a shorter element in $F$.
Let $F_1=\{ 3241 \}$, then $\sigma=6251473$ $2$-avoids $(F_1,G)$, but $2$-contains $(F,G)$.  Thus 
$\Av_2(F,G)\neq \Av_2(F_1,G)$.\end{example}

\begin{example}
\label{egB}
If 	$F=\{4123, 4231, 43251,3241 \}$ and
$G=\{4{1}352 \}$, 
then we claim that the pattern $43251\in F$ {\em is} redundant.  Suppose $\sigma$ $2$-contains $(F,G)$. Either this is because of $\beta\in F\setminus\{43251\}$, or not. If we suppose not, then $\sigma$ must $2$-contain $(F,G)$ because of $43251$.
 If $\sigma$ contains $43251$ then it contains $3241$, so each of the subpermutations $\gamma \in\{ 4351,  4251 , 3251 \} $ must be saved by $4{1}352\in G$ 
	(otherwise we would say $\sigma$ $2$-contains $(F,G)$ because of some $\beta\in F\setminus\{43251\}$).
Thus, $\sigma$ must  have subpermutation $4a3b251$ such that $4a351 \sim 4a251 \sim 3b251 \sim (4{1}352)$, so $a,b<1$. 
So, $4a3b251$ is either order isomorphic to  $403(-1)251$ (so $\sigma$ contains $4231$) or $4(-1) 3 0 251$ (so $\sigma$ contains  $4123$). So, $\sigma$ $2$-contains $F$ because of $4123$  or $4231$ which contradicts that it is because of $43251$ only. Thus, $\sigma$  $2$-contains $(F \setminus \{ 43251 \},G)$. Conversely if  $\sigma$  $2$-contains $(F \setminus \{ 43251 \},G)$ then it clearly $2$-contains $(F,G)$.
Thus $\Av_2(F,G)=\Av_2(F \setminus \{ 43251 \},G)$.
\end{example}

However, we can state some general rules for removing redundant elements of $F$ or $G$.
\begin{restatable}{lemma}{lemA}
\label{lem:removeA}
If $\alpha\in G$ and for all $\gamma\in F$, $\gamma\not\containedin \alpha$, then $\Av_2(F,G)=\Av_2(F,G\setminus\{\alpha\})$.
\end{restatable}

\begin{restatable}{lemma}{lemB}
\label{lem:removeB}
If $\alpha,\beta \in G$, $\beta\containedin \alpha$, $\beta\neq\alpha$,
and for all $\kappa\in F$, $\kappa\containedin\alpha$ implies $\kappa\containedin \beta$, then $\Av_2(F,G)=\Av_2(F,G\setminus\{\alpha\})$.
\end{restatable}

\begin{restatable}{lemma}{lemC}
\label{lem:removeC}
If $\kappa,\lambda\in F$, $\kappa\containedin \lambda$ 
and for all $\alpha\in G$, $\kappa\not\containedin \alpha$, then $\Av_2(F,G)=\Av_2(F\setminus\{\lambda\},G)$.
\end{restatable}
Proofs can be found in \cref{appendix}.
Certainly we believe that further lemmas could be stated and proved to remove more redundant elements.
For the purpose of this paper we do not pursue this, we content ourselves to find a characterisation for $k$-pop stack sortable permutations in terms of finite sets only.

\section{Blocks}
Let $\sigma$ be a permutation. Call a factor $B_i=a_{i,1}a_{i,2}\dots a_{i,n_i}$ of  $\sigma$ a  {\em block} if $n_i>0$ and $a_{i,j}>a_{i,j+1}$ for all $1\leq j<n_i$. (Recall that factor means the entries are contiguous in $\sigma$.)
A {\em (maximal) block decomposition} of $\sigma$ is an expression of the form 
$\sigma=B_1B_2B_3\dots B_m$ where each $B_i$ is a block and for any two adjacent blocks $B_{i} = a_{i,1}a_{i,2}\dots a_{i,n_i}$ and $B_{i+1}=a_{i+1,1}a_{i+1,2}\dots a_{i+1,n_{i+1}}$ we have $a_{i,n_i} < a_{i+1,1}$.
For example   $\sigma={8}7{6}3452{1}$ has block decomposition $B_1= {8}7{6}3, B_2=4, B_3=52{1}$. For convenience we indicate the block decomposition of $\sigma$ by inserting $\mid$ symbols to separate blocks, so for our example we write ${8}7{6}3\mid4\mid52{1}$.  

If $B_i=a_{i,1}a_{i,2}\dots a_{i,n_i}$ is a block, let  $\widetilde{B_i}=a_{i,n_i}\dots a_{i,2}a_{i,1}$.
We have the following.
\begin{lemma}[\cite{RebeccaTwoPass}]
	If $\sigma$ has block decomposition $B_1B_2B_3\dots B_m$ then \[p_1(\sigma)= \widetilde{B_1}\widetilde{B_2}\widetilde{B_3}\dots \widetilde{B_m}.\]
\end{lemma}

\noindent
For example $\sigma=987354621=9873\mid 54\mid 621$ so  $p_1(\sigma)=3789\, 45\, 126$. 

\begin{lemma}[\cite{MR3940979}]\label{lem3}
Let $\sigma$ be a permutation. Then each block in the block decomposition of $p_1(\sigma)$ contains at most 3 tokens.\end{lemma}

\noindent
For example, if  $\sigma=52341=52\mid 3\mid41$ then $p_1(\sigma)=25314=2\mid 531\mid 4$.

\begin{lemma}
	\label{4k_blocks}
	Let $\sigma$ be a permutation with block decomposition  $B_1B_2B_3\dots B_m$, 
	$a\in B_{i+1}, b\in B_{i+n}$ 
	two entries of $\sigma$ with $a>b$, $n \geq1$.  (See for example \cref{fig:long-blocks}). 
   	If $n\geq 3^{k}$ then $\sigma$ is not $k$-pass pop stack sortable.
\end{lemma}
\begin{proof}
We proceed by induction. For $k=0$, $\sigma$ is not sortable by $0$ passes since $a>b$.
Assume the statement is true for $k\in\N$, and consider $\sigma= B_1\dots B_{i+1}\dots B_{i+n}\dots B_m$, $a\in B_{i+1}, b\in B_{i+n}$ with $n\geq 3^{k+1}$.  Colour the tokens $a,b$ and all tokens in  $B_{i+2},\dots, B_{i+n-1}$ 
\textbf{bold}. Then the number of bold tokens in $\sigma$ is at least $n\geq 3^{k+1}$ since each block contains at least one token. After one pass, each block of $p_1(\sigma)$ 
can contain at most $3$ tokens, 
so the number of blocks with bold entries  is at least $\frac{3^{k+1}}3=3^k$. Since $p_1(\sigma)= \widetilde{B_1}\dots \widetilde{B_{i+1}}\widetilde{B_{i+2}}\dots \widetilde{B_{i+n-1}}\widetilde{B_{i+n}}\dots \widetilde{B_m}$ and $\widetilde{B_{i+1}},\widetilde{B_{i+n}}$ contain just one bold token each ($a,b$ respectively),  we have that $a$ must be in the first block with bold entries in the block decomposition of $p_1(\sigma)$, and $b$ in the last, so by inductive hypothesis $p_1(\sigma)$  cannot be sorted by $k$ passes. Thus $\sigma$ cannot be sorted by $k+1$ passes.\end{proof}

We remark that the bound of $3^{k}$ in the preceding statement is an extreme overestimate\footnote{for example, if each block had exactly one token, the factor would be an increasing sequence.}, a more careful  argument should prove a linear bound rather than exponential. However for the purpose of this paper any bound suffices.

\section{General characterisation of $k$-pass pop stack sortable permutations }
\label{sec:k-pass pop stack characterisation}

\begin{theorem}\label{thm:main}	Let $k\in \N_+$.
	There exists a pair of finite sets $(F_k, G_k)$  such that the set of all  $k$-pass pop stack sortable permutations is equal to $\Av_2(F_k, G_k)$.  Moreover, the sets $F_k,G_k$ can be algorithmically constructed. 
\end{theorem}
\begin{proof}
	Let $S_k$ denote  the set of all  $k$-pass pop stack sortable permutations.
	We proceed by induction, with the base case $k=1$ established by Avis and Newborn \cite{MR624050} (specifically, $F_1=\{ 231, 312 \}, G_1=\emptyset$).
	Assume $F_{k-1},G_{k-1}$ have been  constructed,  are finite, and $S_{k-1}=\Av_2(F_{k-1}, G_{k-1})$.

	Let 
	$ f_{\max}=\max\{|\beta|\mid \beta\in F_{k-1}\}$ and 
	$C= 3^{k+2}f_{\max}$. 
	Then define  \[\begin{array}{lll}
	\Omega_1  &=&  \{\tau\in S^\infty\mid  |\tau|\leq 3f_{\max}, \tau\not\in S_k\}, \\\Omega_2&=&
	\{\kappa\in S^\infty\mid |\kappa|\leq C,\kappa\in S_k, \exists \tau\in \Omega_1[\tau\containedin \kappa]
	\}.\end{array}\] 
	
	 \noindent
	Claim 1: $\Omega_{i}$ are both finite:  both are subsets of the set of all permutations of length at most $C$.
	\\
	
	\noindent
	Claim 2: $\Omega_{i}$ are algorithmically constructible: 
	we only have finitely many $\tau,\kappa$ of length at most $3f_{\max}, C$ respectively, for each $\tau, \kappa$ we can check  $\tau\not\in S_k$ or $\kappa\in S_k$  in linear time by passing them according to the deterministic procedure, and we can check all subpermutations of $\kappa$ of length at most $3f_{\max}$ to see whether or not at least one has reduced form  in $\Omega_1$. \\

	\noindent
	Claim 3:  $\sigma\not\in S_k$ if and only if  $\sigma$ $2$-contains $(\Omega_1,\Omega_2)$.

	\begin{proof}[Proof of Claim 3]
	We have   $\sigma\not\in S_k$  if and only if  $p_1(\sigma)\not\in S_{k-1}$ if and only if  $p_1(\sigma)$ $2$-contains $(F_{k-1},G_{k-1})$.
		
		To prove the forward direction,
		suppose  $p_1(\sigma)$ $2$-contains $(F_{k-1},G_{k-1})$, and further assume this is because of some $\zeta\subperm p_1(\sigma)$ with $\red(\zeta)\in F_{k-1}$, and there is no 
		 $\delta\subperm \sigma$ with $ \zeta\subperm \delta$ and $\red(\delta)\in G_{k-1}$.
		Note, there may be  many choices of $\zeta$ to take, but  fix  one choice.
		\begin{enumerate}\item 
			Mark tokens corresponding to $\zeta$ in $p_1(\sigma)$  \textbf{bold}.  
			Let $\zeta'\subperm \sigma$ be such that after one pass, the tokens belonging to $\zeta'$ are the bold tokens corresponding to $\zeta\subperm p_1(\sigma)$. Mark the $\zeta'$ tokens bold as well.
			Note that $ |\zeta'|\leq f_{\max}$.

			\medskip
			
			For example, if $\sigma=987354621$ then $p_1(\sigma)=378945126$ which $2$-contains $(\{231, 312\},\emptyset)$ because of (for instance)
			the  subpermutation $\zeta=956\sim 312$ of  $p_1(\sigma)$. We write  $p_1(\sigma)$ as  $378\textbf{9}4\textbf{5}12\textbf{6}$, and thus $\sigma$ as  $\textbf{9}873\textbf{5}4\textbf{6}21$.

			\medskip
			
			\item   
			Next, write  $\sigma$ in block decomposition $\sigma=B_1B_2B_3\dots B_m$ with $B_i=a_{i,1}\dots a_{i,n_i}$. Say that $B_i$ is \textbf{bold} if it contains at least one bold entry (from $\zeta'$).
			We wish to delete non-bold entries of $\sigma$ but we do not want to {\em merge} bold blocks, so we apply the following subroutine. 
			
			\begin{itemize}\item[--]
				set $\tau=\sigma$\item[--]
				while $a_{i,j}$ is a non-bold letter,   \begin{itemize}\item[] if removing $a_{i,j}$ from $\tau$
					does not cause two or more bold blocks to merge, delete $a_{i,j}$ from $\tau$.
					
				\end{itemize}

			\end{itemize}
			
			We claim that at the end of this process $|\tau|\leq 3|\zeta'|\leq 3f_{\max}$. Let $a_{i,j}\in B_i$ with $1<i<m$ be a non-bold token in $\tau$. If at most one of $B_{i-1},B_i,B_{i+1}$ is bold,  then removing $a_{i,j}$ cannot merge bold blocks. Else assume at least two of $B_{i-1},B_i,B_{i+1}$ are bold. If  $a_{i,j}$ is not the first or last entry in $B_i$, it can be deleted without merging blocks. This leaves at most two unbold entries in $B_i$. For $B_1$ (resp. $B_m$)  we  can delete all except the last (resp. first) entry without merging bold blocks. This leaves at most two unbold entries in each block.
			Then in the worst case each block contains just one bold entry, with an unbold entry on either side. For example, if we get to  $\tau=\dots \mid 12, \mathbf{10},  8 \mid 9\mathbf{7}5\mid 6\mathbf{4}2\mid 3\mathbf{1} $ then we cannot delete $8,9, 5,6,2,3$ without merging blocks.
			
			\item
			After this, we obtain a permutation $\tau\subperm \sigma$ such that the bold letters $\zeta'\subperm \tau$ and $|\tau|\leq 3f_{\max}$. 
			
		\end{enumerate}

		We now claim that $p_1(\tau)$ $2$-contains $(F_{k-1},G_{k-1})$ because of the same subpermutation $\zeta$ of $p_1(\sigma)$. Since bold blocks are preserved in $\tau$, we know that $p_1(\tau)$ will also contain $\zeta$.
		Now suppose there is some $ \delta\subperm p_1(\tau)\left(\subperm p_1(\sigma)\right)$ with $\red(\delta) \in G_{k-1}$ and $ \zeta'\subperm \delta$. This means that the same $\delta$ saves $(\zeta,\sigma)$, which contradictions our original assumption. Thus  $p_1(\tau)$ $2$-contains $(F_{k-1},G_{k-1})$ which implies $p_1(\tau)\not\in S_{k-1}$ which implies $\tau\not\in S_k$.
		
		Thus since $|\tau|\leq 3 f_{\max}$ and $\tau\not\in S_k$, we have $\tau\in \Omega_1$ by definition.  To finish this direction, we will show that $\tau$ is not saved by any subpermutation order-isomorphic to something in $\Omega_2$.

		Suppose (for contradiction) that there is some  $\delta\subperm \sigma$ with $\tau\subperm \delta$ and $\red(\delta)\in \Omega_2$.
		This means $\delta\in S_k$ so $p_1(\delta)\in S_{k-1}$ so  $p_1(\delta)$ $2$-avoids $(F_{k-1},G_{k-1})$.

		Now $p_1(\delta)$ will  contain $\zeta$ since blocks containing $\zeta'\subperm \tau\subperm \delta$ will not merge after one pass.
		Since $p_1(\delta)$ $2$-avoids $(F_{k-1},G_{k-1})$ and contains $\zeta$, 
		there must be some   $\alpha\in G_{k-1}$ which saves $\zeta\subperm p_1(\delta)$. 
		This means there is some $\alpha'\subperm p_1(\delta)$ with $\zeta\subperm \alpha'$ and $\alpha'\sim\alpha$.

		We claim $\alpha\in G_{k-1}$ also saves $\zeta\subperm p_1(\sigma)$, since there exists $\alpha'\subperm p_1(\delta)\subperm p_1(\sigma)$ with $\zeta\subperm \alpha'$ and $\alpha'\sim\alpha$. This contradicts that $p_1(\sigma)$ $2$-contains $(F_{k-1},G_{k-1})$ because of $\zeta$.
		Thus we have shown $\sigma\not \in S_k$ implies $\sigma$ $2$-contains $(\Omega_1,\Omega_2)$.

		\bigskip
		Now for the converse direction, 
		suppose that  $\sigma$ $2$-contains $(\Omega_1,\Omega_2)$, and so we can assume that this is because of 
		$\gamma\subperm\sigma$ and $\tau\in\Omega_1$ with $\gamma\sim \tau$
		(which is not saved by any $\alpha\in \Omega_2$), and so by definition $\gamma\not\in S_k$ so  $p_1(\gamma)$ $2$-contains $(F_{k-1},G_{k-1})$. 

		Assume (for contradiction) that $\sigma\in S_k$. We will show that this implies we can construct some $\kappa\subperm\sigma$ such that $\kappa\in S_k, \gamma\subperm \kappa $ and $|\kappa| \leq C$.
		 If so, then we can construct $\alpha\in \Omega_2$ with $\alpha\sim\kappa$ and $\red(\gamma)\containedin \alpha$, which means $\alpha$ saves $\gamma\subperm\sigma$, and this gives a contradiction that $\sigma$ $2$-contains $(\Omega_1,\Omega_2)$ because of $\gamma$.
		
		Here is how we construct $\kappa$. 
		In $\sigma$, mark the tokens corresponding to $\gamma$  \textbf{bold}.

		 Call a block of $\sigma$ bold if it contains at least one bold token, and otherwise a block is called non-bold.
		Starting with  $\kappa=\sigma$, we  delete non-bold tokens using the following procedure, which is more careful than the  subroutine used in the proof of the forward direction above. The goal is to delete non-bold tokens to obtain a permutation $\kappa$ with subpermutation $\gamma$ such that for {\em every} block $B$ in $\kappa$ there is a block $B'$ in $\sigma$ so that  the tokens in $B$ are tokens in $ B'$. That is, we do not allow {\em any} (bold or unbold) blocks to merge, only to be deleted entirely.
		
		\begin{itemize}\item[--] set $\kappa=\sigma$\item[--]
			while $a_{i,j}\in B_i$ is a non-bold letter,   \begin{itemize}\item[-] if removing $a_{i,j}$ from $\kappa$
				does not cause two or more  blocks {\em of any kind} (bold or non-bold) in $\kappa$ to merge, delete $a_{i,j}$ from $\kappa$,
				\item[-] if  $B_i$ is non-bold and removing the entire  block $B_i$ at once does not cause any of the remaining blocks to merge, then delete $B_i$.
			\end{itemize}
		\end{itemize}
		We claim that at the end of this process each block contains at most two non-bold entries, which will be the first and last entries of the block. However, since we have not deleted non-bold blocks if their removal would cause other blocks to merge, we could have arbitrarily long factors of non-bold blocks, as in the example in  Figure~\ref{fig:long-blocks}.

  \begin{figure}[ht]

\begin{center}
      \begin{tikzpicture}[scale = .5, auto = center, inner sep=.3mm]

     \draw[color=gray] (6.5,19) -- (8.5,19) -- (8.5,23) -- (6.5,23) -- (6.5,19);
            \node[circle, fill=black] at (7,22) {.};
              \node[circle, fill=gray] at (8,20) {};
              
                   \draw[color=gray] (8.5,17) -- (10.5,17) -- (10.5,22) -- (8.5,22) -- (8.5,17);
                        \node[circle, fill=gray] at (9,21) {};
              \node[circle, fill=gray] at (10,18) {};

                                 \draw[color=gray] (10.5,15) -- (12.5,15) -- (12.5,20) -- (10.5,20) -- (10.5,15);
                        \node[circle, fill=gray] at (11,19) {};
              \node[circle, fill=gray] at (12,16) {};

                                 \draw[color=gray] (12.5,13) -- (14.5,13) -- (14.5,18) -- (12.5,18) -- (12.5,13);
                        \node[circle, fill=gray] at (13,17) {};
              \node[circle, fill=gray] at (14,14) {};
              
                                 \draw[color=gray] (14.5,11) -- (16.5,11) -- (16.5,16) -- (14.5,16) -- (14.5,11);
                        \node[circle, fill=gray] at (15,15) {};
              \node[circle, fill=gray] at (16,12) {};
                                          
                         \draw[color=gray] (16.5,9) -- (18.5,9) -- (18.5,14) -- (16.5,14) -- (16.5,9);
                        \node[circle, fill=gray] at (17,13) {};
              \node[circle, fill=gray] at (18,10) {};

                                                             \draw[color=gray] (18.5,7) -- (20.5,7) -- (20.5,12) -- (18.5,12) -- (18.5,7);
                        \node[circle, fill=gray] at (19,11) {};
              \node[circle, fill=gray] at (20,8) {};

                                                             \draw[color=gray] (20.5,6) -- (22.5,6) -- (22.5,10) -- (20.5,10) -- (20.5,6);
                        \node[circle, fill=gray] at (21,9) {};
              \node[circle, fill=black] at (22,7) {.};

      \end{tikzpicture}
      \end{center}

\caption{ $\kappa=\mathbf{16}, 14, 15, 12, 13, 10,11, 8 9 6 7 4 5 2 3 \mathbf{1}$ }
\label{fig:long-blocks}
  \end{figure}
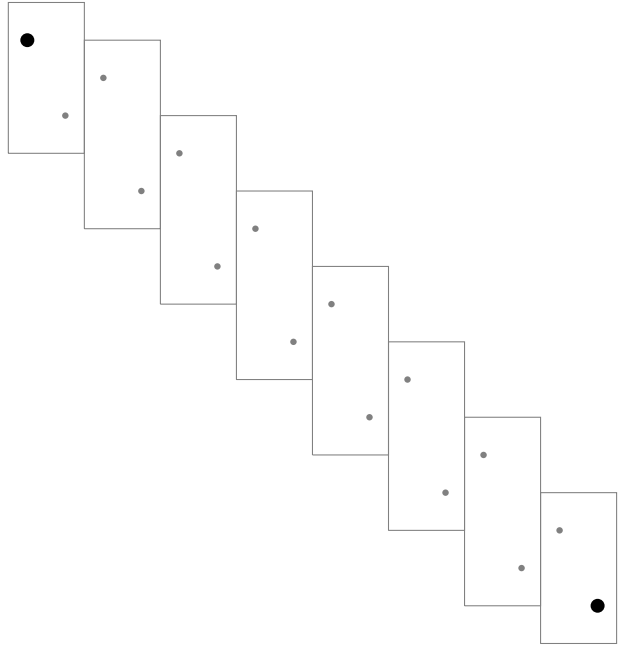

		Suppose $\kappa$ has block decomposition $T_1\dots T_m$, and $\kappa$ contains two bold entries $a,b$ with $a\in T_{i+1}$ and $b\in T_{i+n}$ with $T_{i+2}, \dots, T_{i+n-1}$ having no bold entries. (Note that each $T_j$ is a subset of some $B_k$ in the block decomposition of $\sigma$,  by construction.)
		If $a<b$ then the subroutine is not complete: we could delete $T_{i+2}, \dots, T_{i+n-1}$ completely without merging $T_{i+1}$ with $T_{i+n}$ since the last entry of $T_{i+1}$ is smaller than the first entry of $T_{i+n}$.
		Therefore since we assume the subroutine is complete, we know that $a>b$.
If $n\geq 3^k$ then by Lemma~\ref{4k_blocks} $\sigma$ cannot be sorted by $k$ passes, which is a contradiction.
		Thus we have $|\kappa|$ is at most $3|\gamma|$ (the bold blocks with a non-bold entry first and last) plus $2.3^k|\gamma|$ ($3^k$ factors each containing $2$ tokens, as in Figure~\ref{fig:long-blocks}, in the worst case occurring between every pair of bold tokens from $\gamma$). Thus \[|\kappa|\leq (3+2.3^k)|\gamma|\leq (3+2.3^k)3f_{\max}\leq 3^{k+2}f_{\max}=C\]
		
		If $\kappa$ cannot be sorted, then 
		 $p_1(\kappa) \not \in S_{k-1}$ because of some subpermutation $\tau$ with $\red(\tau)\in F_{k-1}$, but  since no block has merged in obtaining $p_1(\kappa)$, $p_1(\sigma)$ also  contains $\tau$ which cannot be saved since blocks containing the tokens forming $\tau$ are fixed. Thus $p_1(\sigma)\not\in S_{k-1}$, so 
		 $\sigma \not \in S_{k}$, contradiction.
			\end{proof}

	Claim 3 implies that we could take $F_k=\Omega_1, G_k=\Omega_2$ and the theorem is done. However, 
	we can first  apply \cref{lem:removeB,lem:removeC}  \footnote{Note that by construction there is no need to apply Lemma~\ref{lem:removeA}.
} to obtain smaller sets with the same $2$-avoidance set, so we will call the result of applying these (in some order) $(F_k,G_k)$. As remarked in Subsection~\ref{subsec:redundant}, the result of applying the lemmas is not guaranteed to give a set that is minimal or unique. Note that each of these lemmas needs to check a finite set so each is algorithmic.\end{proof}

Applying the construction in the proof above to the case $k=2$ yields a much larger pair of sets than those appearing in \cref{thm:2pop} given by Pudwell and Smith\footnote{$f_{\max}=3$ and $k=2$  so $C=3^5=243$.}. The point of our argument is that it is general; our upper bounds can certainly  be lowered, and more lemmas to reduce the size of the sets $F,G$ could be proved to sharpen the result. 
See \cite{GohThesis} for a discussion  of explicit avoidance sets for the case $k=3$.

\section{Outlook}

Our notion of $2$-containment opens up some interesting possibilities. Recall that by  
Kaiser-Klazar \cite[Thm. 3.4]{Klazar} and Marcus-Tardos \cite{MarcusT}
 the function counting the number of permutations of length $n$ in any $\Av(F)$ for $F$ non-empty is either polynomial or exponential.
It is conceivable some pair of (finite or infinite) sets $(F,G)$ could have 2-avoidance set with growth function strictly between polynomial and exponential, or between exponential and factorial. Example~\ref{eg:egFactorial} shows non-trivial 2-avoidance sets with factorial growth.

Generating functions for $2$-avoidance sets might also exhibit  interesting behaviour. 
For the sets $(F_k,G_k)$ in \cref{thm:main}
 we know by \cite{MR3940979} the generating functions 
  are rational for all $k$, but for general sets $F,G$ the set  $\Av_2(F,G)$ could have interesting enumerations.

\bibliographystyle{plain}
\bibliography{refs}

\appendix
\section{Proofs of \cref{lem:removeA,lem:removeB,lem:removeC}}
\label{appendix}

We start with the following observation.
\begin{lemma}\label{lemFirst}
Let $F_1,F,G_1,G\subseteq S^\infty$ with $F_1\subset F, G_1\subset G$. Then 
\begin{itemize}\item[--] $\Av_2(F,G_1)\subseteq \Av_2(F,G)$ and 
\item[--] $\Av_2(F,G)\subseteq \Av_2(F_1,G)$.
\end{itemize}
\end{lemma}
\begin{proof}
If $\sigma$ $2$-contains $(F,G)$ then $\exists \gamma\subperm \sigma$ with $\red(\gamma)\in F$ and no $\gamma\subperm \delta\subperm \sigma$ with $\red(\delta)\in G$, so in particular, no $\gamma\subperm \delta\subperm \sigma$ with $\red(\delta)\in G_1$, so $\sigma$ $2$-contains $(F,G_1)$. 
So $\Av_2(F, G_1)\subseteq \Av_2(F, G)$.

On the other hand, if $\sigma$ $2$-contains $(F_1, G)$ then  $\exists \gamma\subperm \sigma$ with $\red(\gamma)\in F_1$ and no $\gamma\subperm \delta\subperm \sigma$ with $\red(\delta)\in G$, so in particular $\red(\gamma)\in F$, so $\sigma$ $2$-contains 
 $(F,G)$. So $\Av_2(F, G)\subseteq \Av_2(F_1, G)$. 
 \end{proof}

 \lemA*
\begin{proof}
By \cref{lemFirst} it suffices to show $\Av_2(F, G)\subseteq \Av_2(F, G\setminus\{\alpha\})$. 
If $\sigma$ $2$-contains $(F,G\setminus\{\alpha\})$ then $\exists \gamma\subperm \sigma$ with $\red(\gamma)\in F$ and no $\gamma\subperm \delta\subperm \sigma$ with $\red(\delta)\in G\setminus\{\alpha\}$, and since $\gamma\not\containedin \alpha$ then $\red(\delta)\neq \alpha$ either, so we have that there can be no $\gamma\subperm \delta\subperm \sigma$ with $\red(\delta)\in G$, so $\sigma$ $2$-contains $(F,G)$, and so   $\Av_2(F, G)\subseteq \Av_2(F, G\setminus\{\alpha\})$. 
 \end{proof}

\lemB*
\begin{proof}
By \cref{lemFirst} it suffices to show $\Av_2(F, G)\subseteq \Av_2(F, G\setminus\{\alpha\})$. 
If $\sigma$ $2$-contains $(F,G\setminus\{\alpha\})$ then $\exists \gamma\subperm \sigma$ with $\red(\gamma)\in F$ and no $\gamma\subperm \delta\subperm \sigma$ with $\red(\delta)\in G\setminus\{\alpha\}$. If $\gamma\subperm \delta'\subperm \sigma$ with $\red(\delta')=\alpha$, then $\gamma\containedin \alpha$ which implies $\gamma\containedin \beta$, which contradicts that no $\delta$ with $\red(\delta)=\beta\in G\setminus\{\alpha\}$ can exist with $\gamma\subperm \delta\subperm \delta'\subperm \sigma$. Thus $\sigma$ $2$-contains $(F,G)$, and so   $\Av_2(F, G)\subseteq \Av_2(F, G\setminus\{\alpha\})$. 
 \end{proof}
 
 \lemC*
\begin{proof}
By \cref{lemFirst} it suffices to show $\Av_2(F\setminus\{\lambda\}, G) \subseteq\Av_2(F, G)$. 
If $\sigma$ $2$-contains $(F,G)$ then $\exists \gamma\subperm \sigma$ with $\red(\gamma)\in F$ and no $\gamma\subperm \delta\subperm \sigma$ with $\red(\delta)\in G$.  If $\red(\gamma)=\lambda$, then $\exists \gamma'\subperm \sigma$ with $\red(\gamma')=\kappa$ since $\kappa\containedin \lambda$. Since $\kappa\not\containedin\alpha$ for all $\alpha\in G$, then there is  no $\delta\subperm \sigma$ with  $\gamma'\subperm \delta$ and $\red(\delta)\in G$. 
Thus $\sigma$ $2$-contains $(F\setminus\{\lambda\},G)$, and so   $\Av_2(F\setminus\{\lambda\}, G) \subseteq\Av_2(F, G)$. 
 \end{proof}

\end{document}